\documentclass[10pt]{article}
\usepackage{amsmath, amsthm, amssymb}
\usepackage{fullpage}
\usepackage{verbatim}

\newtheorem{theorem}{Theorem}[section]

\newtheorem{lemma}[theorem]{Lemma}
\newtheorem{prop}[theorem]{Proposition}
\newtheorem{conj}[theorem]{Conjecture}

\begin{document}

\title{Italian Domination of Cartesian Products of Directed Cycles}
\author{Christopher M. van Bommel\footnote{Supported by a Pacific Institute for the Mathematical Sciences Post-doctoral Fellowship.}\\Department of Mathematics\\University of Manitoba\\Winnipeg, MB, Canada\\\texttt{Christopher.VanBommel@umanitoba.ca}}

\maketitle

\begin{abstract}
An Italian dominating function on a (di)graph $G$ with vertex set $V(G)$ is a function $f: V(G) \to \{0, 1, 2\}$ such that every vertex $v \in V(G)$ such that $f(v) = 0$ has an (in)neighbour assigned 2 or two (in)neighbours assigned 1.  We complete the investigation of the Italian domination numbers of Cartesian products of directed cycles.
\end{abstract}

\section{Introduction}

A \emph{dominating set} $S$ of a graph $G$ is a subset of the vertices of $G$ such that every vertex not in $S$ has a neighbour in $S$.  The \emph{domination number} $\gamma(G)$ is the minimum cardinality of a dominating set in $G$.  A \emph{2-dominating set} $S$ of a graph $G$, introduced by Fink and Jacobson~\cite{FJ85}, is a subset of the vertices of $G$ such that every vertex not in $S$ has two neighbours in $S$.  The \emph{2-domination number} $\gamma_2(G)$ is the minimum cardinality of a 2-dominating set in $G$.

A \emph{Roman dominating function} of a graph $G$, introduced by Cockayne et al.~\cite{CDHH04} is a function $f: V(G) \to \{0, 1, 2\}$ such that for every $v \in V(G)$ with $f(v) = 0$, it is adjacent to a vertex $w$ with $f(w) = 2$.  The \emph{Roman domination number} $\gamma_R(G)$ is the minimum weight, i.e.\ $\sum_{v \in V(G)} f(v)$, of a Roman dominating function on $G$.  A \emph{weak Roman dominating function} of a graph $G$, introduced by Henning and Hedetniemi~\cite{HH03} is a function $f: V(G) \to \{0, 1, 2\}$ such that for every $v \in V(G)$ with $f(v) = 0$, it is adjacent to a vertex $w$ with $f(w) \neq 0$, and the derived function $g : V(G) \to \{0, 1, 2\}$ with $g(v) = 1$, $g(w) = f(w) - 1$, and $g(x) = f(x)$ for all $x \neq v, w$ is such that $g$ is a dominating function.  The \emph{weak Roman domination number} $\gamma_r(G)$ is the minimum weight of a weak Roman dominating function.

A \emph{2-rainbow dominating function} of a graph $G$, introduced by Bre\v{s}ar et al.~\cite{BHR08}, is a function $f: V(G) \to \mathcal{P}(\{1, 2\})$ such that for every $v \in V(G)$ with $f(v) = \emptyset$, it is adjacent to a vertex $w$ with $1 \in f(w)$ and it is adjacent to a vertex $x$ with $2 \in f(x)$.  The \emph{2-rainbow domination number} $\gamma_{r2}(G)$ is the minimum weight, i.e.\ $\sum_{v \in V(G)} |f(v)|$, of a 2-rainbow dominating function.

A \emph{\{2\}-dominating function} of a graph $G$, introduced by Domke et al.~\cite{DHLF91}, is a function $f: V(G) \to \{0, 1, 2\}$ such that for every $v \in V(G)$, its closed neighbourhood has weight at least two.  The \emph{\{2\}-domination number} $\gamma_{\{2\}}(G)$ is the minimum weight of a \{2\}-dominating function.

An \emph{Italian dominating function}, or \emph{Roman \{2\}-dominating function} of a graph $G$, introduced by Chellali et al.~\cite{CHHM15} is a function $f: V(G) \to \{0, 1, 2\}$ such that for every $v \in V(G)$ with $f(v) = 0$, it is adjacent to a vertex $w$ with $f(w) = 2$ or two vertices $x_1, x_2$ such that $f(x_1) = f(x_2) = 1$.  The \emph{Italian domination number} $\gamma_I(G)$ is the minimum weight of a Italian dominating function.

Chellali et al.~\cite{CHHM15} derived the following chains of inequalities relating these domination parameters.

\begin{theorem} \cite{CHHM15}
For every graph $G$,
\[
\gamma(G) \le \gamma_r(G), \le \gamma_I(G), \gamma_{r2}(G) \le \gamma_R(G) \le 2 \gamma(G), \qquad \gamma_I(G) \le \gamma_2(G).
\]
\end{theorem}

Moreover, Chellali et al.~\cite{CHHM15} demonstrated the following sharpness results for certain classes of graphs.  A graph is a \emph{cactus} if it is connected and any two of its cycles have at most one vertex in common.

\begin{theorem} \cite{CHHM15}
For every tree $T$, $\gamma_I(T) = \gamma_{r2}(T)$.
\end{theorem}

\begin{theorem} \cite{CHHM15}
If $G$ is a cactus graph with no even cycle, then $\gamma_I(G) = \gamma_{r2}(G)$.
\end{theorem}

Finally, Chellali et al.~\cite{CHHM15} determined that the decision problem for Italian domination is NP-complete, even when restricted to bipartite graphs.

Klostermeyer and MacGillivray~\cite{KM19} demonstrated the following lower bound on the Italian domination number of trees.

\begin{theorem}
If $T$ is a tree with at least two vertices, then $\gamma_I(T) \ge \gamma(T) + 1$.
\end{theorem}

Henning and Kloystermeyer~\cite{HK17} characterized the trees for which $\gamma_I(T) \ge \gamma(T) + 1$ and for which $\gamma_I(T) = 2 \gamma(T)$.

The Italian domination number has been studied for Cartesian products of cycles by multiple authors.  Li et al.~\cite{LSX18} determined the following results based on the weak \{2\}-domination number.

\begin{theorem} \cite{LSX18}
For $n \ge 3$, 
\[
\gamma_I(C_n \Box C_3) = \begin{cases}
n, & n \equiv 0 \pmod{3}; \\
n + 1, & n \not\equiv 0 \pmod{3}.
\end{cases}
\]
\end{theorem}

\begin{theorem} \cite{LSX18}
For $n \ge 4$,
\[
\gamma(C_n \Box C_4) = \begin{cases}
\left \lceil \frac{3n}{2} \right \rceil, & n \equiv 0, 1, 3, 4, 5 \pmod{8}; \\
\left \lceil \frac{3n}{2} \right \rceil + 1, & n \equiv 2, 6, 7 \pmod{8}.
\end{cases}
\]
\end{theorem}

St\c{e}pie\'{n} et al.~\cite{SSSZ14} obtained the following based on the 2-rainbow domination number.

\begin{theorem} \cite{SSSZ14}
For $n \ge 5$, $\gamma_I(C_n \Box C_5) = 2n$.
\end{theorem}

Gao et al.~\cite{GWLY20} determined the following general bounds for Cartesian products of cycles.

\begin{theorem} \cite{GWLY20}
For $m \equiv n \equiv 0 \pmod{3}$, $\gamma_I(C_n \Box C_m) = \frac{mn}{3}$.
\end{theorem}

\begin{theorem} \cite{GWLY20}
For $m \not\equiv 0 \pmod{3}$ or $n \not\equiv 0 \pmod{3}$,
\[
\left \lceil \frac{nm}{3} \right \rceil \le \gamma_I(C_n \Box C_m) \le \left \lfloor \frac{2mn + n + 2m + 1}{6} \right \rfloor.
\]
\end{theorem}

Volkmann~\cite{V20} initated the study of the Italian domination number in digraphs.  For a digraph $D$, we require that for every $v \in V(D)$ with $f(v) = 0$, it has an inneighbour $w$ with $f(w) = 2$ or two inneighbours $x_1, x_2$ such that $f(x_1) = f(x_2) = 1$.  We let the maximum out-degree of a digraph be denoted by $\Delta^+(D)$ and the maximum in-degree of a digraph be denoted by $\Delta^-(D)$.  Volkmann~\cite{V20} provided the following preliminary results.

\begin{prop} \cite{V20}
Let $D$ be a digraph of order $n$.  Then $\gamma_I(D) \ge \left \lceil \frac{2n}{2 + \Delta^+(D)} \right \rceil$.
\end{prop}

\begin{prop} \cite{V20}
Let $D$ be a digraph of order $n$.  Then $\gamma_I(D) \le n$, and $\gamma_I(D) = n$ if and only if $\Delta^+(D), \Delta^-(D) \le 1$.
\end{prop}

\begin{prop} \cite{V20}
If $D$ is a directed path or a directed cycle of order $n$, then $\gamma_I(D) = n$.
\end{prop}

Kim~\cite{K20} considered the Italian domination number of the Cartesian and strong products of directed cycles, and proved the following.

\begin{theorem} \cite{K20}
If $m = 2r$ and $n = 2s$ for some positive integers $r, s$, then $\gamma_I(C_m \Box C_n) = \frac{mn}{2}$.
\end{theorem}

\begin{theorem} \cite{K20}
For an odd integer $n \ge 3$, $\gamma_I(C_2 \Box C_n) = n + 1$.
\end{theorem}

\begin{theorem} \cite{K20}
For an integer $n \ge 3$, $\gamma_I(C_3 \Box C_n) = 2n$.
\end{theorem}

\begin{theorem} \cite{K20}
For positive integers $m, n \ge 2$, $\gamma_I(C_m \otimes C_n) = \left \lceil \frac{mn}{2} \right \rceil$.
\end{theorem}

Kim~\cite{K20} also stated the following conjecture.

\begin{conj} \cite{K20}
For an odd integer $n$, $\gamma_I(C_4 \Box C_n) = 2n + 2$.
\end{conj}

In this work, we verify this conjecture and completely settle the question of the Italian dominating number of the Cartesian product of cycles.

\section{Main Result}

Throughtout this section, we consider the vertex set of the directed graph $C_m \Box C_n$ to be $\mathbb{Z}_m \times \mathbb{Z}_n$, with arcs of the form $(i, j) \to (i + 1, j)$ and $(i, j) \to (i, j + 1)$.  We note that $C_m \Box C_n$ and $C_n \Box C_m$ are isomorphic and will use this fact implicitly in deriving our results.  Our approach is to establish properties regarding the set of vertices assigned zero that hold in at least one $\gamma_I(C_m \Box C_n)$-function.  We first show we can avoid having a line of three vertices assigned zero.

\begin{lemma}
There exists a $\gamma_I(C_m \Box C_n)$-function such that no three consecutive vertices (horizontally or vertically) are assigned 0.
\end{lemma}

\begin{proof}
Suppose by way of contradiction that every $\gamma_I(C_m \Box C_n)$-function has three consecutive vertices assigned 0.  Let $f$ be a $\gamma_I(C_m \Box C_n)$-function such that the number of sets of three consecutive vertices assigned 0 is minimized.  Then there exist consecutive vertices $(i, k)$, $(i + 1, k)$, $(i + 2, k)$ such that $f((i, k)) = f((i + 1, k)) = f((i + 2, k)) = 0$.  Now, since $f$ is a $\gamma_I(C_m \Box C_n)$-function, we require that $f((i + 1, k - 1)) = f((i + 2, k - 1)) = 2$.  Consider the function $g$ given by
\[
g((x, y)) = \begin{cases}
0, & (x, y) = (i + 2, k - 1); \\
1, & (x, y) = (i + 2, k); \\
\max\{1, f((x, y))\}, & (x, y) = (i + 3, k - 1); \\
f((x, y)), & \mbox{otherwise.}
\end{cases}
\]
If $f(i + 3, k - 1) \neq 0$, then $g$ is an Italian dominating function with smaller weight than $f$, a contradiction.  If either $f((i + 2, k - 2)) \neq 0$ or $f((i + 2, k - 3)) \neq 0$, then $g$ is a $\gamma_I(C_m \Box C_n)$-function with fewer sets of three consecutive vertices assigned 0, contradicting the definition of $f$.  Otherwise, we have $f((i + 3, k - 1) = f((i + 2, k - 2)) = f((i + 2, k - 3)) = 0$, and since $f$ is a $\gamma_I(C_m \Box C_n)$-function, we require that $f((i + 1, k - 2)) = 2$.  Now consider the function $h$ given by
\[
h((x, y)) = \begin{cases}
0, & (x, y) = (i + 2, k - 1); \\
1, & (x, y) \in \{(i + 1, k - 2), (i + 2, k - 2), (i + 2, k), (i + 3, k - 1)\}; \\
f((x, y)), & \mbox{otherwise.}
\end{cases}
\]
Then $h$ is a $\gamma_I(C_m \Box C_n)$-function with fewer sets of three consective vertices assigned 0, contradicting the definition of $f$.  The result follows.
\end{proof}

We now show we can further avoid having adjacent vertices assigned zero.

\begin{lemma} \label{0indep}
There exists a $\gamma_I(C_m \Box C_n)$-function such that no pair of adjacent vertices is assigned 0.
\end{lemma}

\begin{proof}
Suppose by way of contradiction that every $\gamma_I(C_m \Box C_n)$-function has a pair of adjacent vertices assigned 0.  Let $f$ be a $\gamma_I(C_m \Box C_n)$-function with no set of three consecutive vertices assigned 0 (whose existence is guaranteed by the previous lemma) such that the number of pairs of adjacent vertices assigned 0 is minimized.  Then without loss of generality, there exist adjacent vertices $(i, k)$, $(i + 1, k)$ such that $f((i, k)) = f((i + 1, k)) = 0$. Then we have $f((i - 1, k)) \neq 0$ and $f((i + 1, k - 1)) = 2$.  

Consider the function $g$ given by
\[
g((x, y)) = \begin{cases}
1, & (x, y) \in \{(i + 1, k - 1), (i + 1, k)\}; \\
f((x, y)), & \mbox{otherwise.}
\end{cases}
\]
If either $f((i + 2, k - 1)) \neq 0$ or $f((i + 2, k - 2)) \neq 0$, then $g$ is a $\gamma_I(C_m \Box C_n)$-function with fewer pairs of adjacent vertices assigned 0, contradicting the definition of $f$.  Otherwise, we have $f((i + 2, k - 1)) = f((i + 2, k - 2))  = 0$.

Now, consider the function $h$ given by
\[
h((x, y)) = \begin{cases}
0, & (x, y) = (i + 1, k - 1); \\
1, & (x, y) \in \{ (i + 1, k), (i + 2, k - 1)\}; \\
f((x, y)), & \mbox{otherwise.}
\end{cases}
\]
If $f((i + 1, k - 2)) \neq 0$, then $h$ is a $\gamma_I(C_m \Box C_n)$-function with fewer pairs of adjacent vertices assigned 0, contradicting the definition of $f$.

Finally, we may assume $f((i + 2, k - 1)) = f((i + 1, k - 2)) = f((i + 2, k - 2)) = 0$.  Then $f((i, k - 2)) \neq 0$ and $f((i + 2, k - 3)) = 2$.  By the arguments above, we may assume that $f((i + 1 + j, k - 1 -2j)) = 2$ for all $j \ge 0$.  Let $S$ be the set of vertices of the form $(i + 1 + j, k - 1 - 2j)$.  

Suppose $S + (0, 1) = S$.  It follows that $|S| \ge mn / 2$, so $f$ has weight at least $mn$.  If $f$ has weight greater than $mn$, then $\mathbf{1}$ is an Italian dominating function of smaller weight, a contradiction.  Otherwise, $\mathbf{1}$ is a $\gamma_I(C_m \Box C_n)$-function with no zeros, contradicting our assumption on $f$.

Otherwise, $S + (0, 1) \cap S = \emptyset$.  Then by construction Consider the function $p$ given by
\[
p((x, y)) = \begin{cases}
1, & (x, y) \in S \cup S + (0, 1); \\
f((x, y)), & \mbox{otherwise.}
\end{cases}
\]
Then $p$ is a $\gamma_I(C_m \Box C_n)$-function with fewer pairs of adjacent vertices assigned 0, contradicting the definition of $f$.  The result follows.
\end{proof}

It therefore follows that there exists a $\gamma_I(C_m \Box C_n)$-function where the vertices assigned zero form an independent set.  We hence establish the independence number $\alpha$ for Cartesian products of cycles. 

\begin{lemma}
\[
\alpha(C_m \Box C_n) = \begin{cases}
\frac{mn}{2}, & m \equiv n \equiv 0 \pmod{2}; \\
\frac{m (n - 1)}{2}, & m \equiv 0, n \equiv 1 \pmod{2}; \\
\frac{m (n - 1)}{2}, & m \equiv n \equiv 1 \pmod{2}, m \ge n.
\end{cases}
\]
\end{lemma}

\begin{proof}
\begin{description}
\item[$m \equiv n \equiv 0 \pmod{2}$:] Suppose $I$ is an independent set of size greater than $mn / 2$.  Then there exists a row containing more than $m / 2$ vertices of $I$.  Then there exists a pair of adjacent vertices of $I$, a contradiction.  Since there exists an independent set of size $mn/2$ by taking all vertices $(i, j)$ such that $i + j \equiv 0 \pmod{2}$, the result follows.

\item[$m \equiv 0, n \equiv 1 \pmod{2}$:] Suppose $I$ is an indepedent set of size greater than $m (n - 1) / 2$.  Then there exists a column containing more than $(n - 1) / 2$ vertices of $I$.  Then there exists a pair of adjacent vertices of $I$, a contradiction.  Since there exists an independent set of size $m (n - 1) / 2$ by taking all vertices $(i, j)$ such that $i + j \equiv 0 \pmod{2}$ and $j \neq n$, the result follows.

\item[$m \equiv n \equiv 1 \pmod{2}$:] Suppose $I$ is an independent set of size greater than $m (n - 1) / 2$.  Then there exists a column containing more than $(n - 1) / 2$ vertices of $I$.  Then there exists a pair of adjacent vertices of $I$, a contradiction.  Since there exists an independent set of size $m (n - 1) / 2$ by taking all vertices $(i, j) = (x + 2y, x)$, $1 \le x \le m$, $1 \le y \le (n - 1) /2$, the result follows.
\end{description}
\end{proof}

We now have the tools to establish our main result.

\begin{theorem} \label{alpha}
\[
\gamma_I(C_m \Box C_n) = \begin{cases}
\frac{mn}{2}, & m \equiv n \equiv 0 \pmod{2}; \\
\frac{m(n+1)}{2}, & m \equiv 0, n \equiv 1 \pmod{2}; \\
\frac{m(n + 1)}{2}, & m \equiv n \equiv 1 \pmod{2}, m \ge n.
\end{cases}
\]
\end{theorem}

\begin{proof}
Let $f$ be a $\gamma_I(C_m \Box C_n)$-function.  By Lemma~\ref{0indep}, there exists a $\gamma_I(C_m \Box C_n)$-function with no pair of adjacent vertices assigned zero.  Hence, $|f| \ge mn - \alpha(C_m \Box C_n)$, which by Lemma~\ref{alpha}, are precisely the values given in the theorem statement.  It remains to establish the existence of an Italian dominating function of this size.  Let $I$ be an independent set of size $mn - \gamma_I(C_m \Box C_n)$ which is shown to exist by Lemma~\ref{alpha}.  Define $f$ as follows:
\[
f((x, y)) = \begin{cases}
1, & (x, y) \notin I; \\
0, & (x, y) \in I.
\end{cases}
\]
It is clear that $f$ is an Italian dominating function, as every vertex in $I$ has its two inneighbours not in $I$ by definition, and hence have value 1.  The result follows.
\end{proof}

\end{document}